\documentclass{conm-p-l}
\numberwithin{equation}{section}
\theoremstyle{plain}
\newtheorem{thm}{Theorem}[section]
\newtheorem{prop}[thm]{Proposition}

\newtheorem{lem}[thm]{Lemma}

\theoremstyle{remark}

\theoremstyle{definition}
\newtheorem{defn}[thm]{Definition}
\newtheorem*{nota}{Notation}

\DeclareFontFamily{OT1}{rsfs}{}
\DeclareFontShape{OT1}{rsfs}{m}{n}{ <-7> rsfs5 <7-10> rsfs7 <10-> rsfs10}{}
\DeclareMathAlphabet{\mycal}{OT1}{rsfs}{m}{n}

\newcommand{\eps}{\varepsilon}
\newcommand{\BJ}[2]{j_{#1} \left(#2 \right)}
\newcommand{\TBJ}[2]{J_{#1} \left(#2 \right)}
\newcommand{\TBJp}[2]{J^\prime_{#1} \left(#2 \right)}
\newcommand{\BJp}[2]{j^\prime_{#1}\left(#2\right)}
\newcommand{\BY}[2]{y_{#1} \left(#2 \right)}
\newcommand{\TBY}[2]{Y_{#1} \left(#2 \right)}
\newcommand{\TBYp}[2]{Y^\prime_{#1} \left(#2 \right)}
\newcommand{\BYp}[2]{y^\prime_{#1}\left(#2\right)}
\newcommand{\BH}[2]{h^{(1)}_{#1}\left(#2 \right)}
\newcommand{\TBH}[2]{H^{(1)}_{#1}\left(#2 \right)}
\newcommand{\BHp}[2]{h^{(1)\prime}_{#1}\left(#2 \right)}
\newcommand{\Rn}{r_n\left(\oeps,\lambda\right)}
\newcommand{\TRn}{R_n\left(\oeps,\lambda\right)}
\newcommand{\Tn}[2]{t_n\left(#1,#2\right)}
\newcommand{\oeps}{\omega_{\eps}}
\newcommand{\lneps}{\left|\ln\eps\right|}
\newcommand{\forget}[1]{ }

\begin{document}

\title[Size Estimates]{On the scattered field generated by a ball inhomogeneity of constant index in dimension three}

\author{Yves Capdeboscq}
\address{Mathematical Institute, 24-29 St Giles, OXFORD OX1 3LB, UK}
\email{yves.capdeboscq@maths.ox.ac.uk}

\thanks{Yves Capdeboscq is supported by the EPSRC Science and Innovation award to the Oxford Centre for Nonlinear PDE (EP/E035027/1).} 

\author{George Leadbetter}
\address{Department of Mathematics, University College London, London WC1E 6BT, UK}

\author{Andrew Parker}
\address{Mathematical Institute, 24-29 St Giles, OXFORD OX1 3LB, UK}


\begin{abstract}
We consider the solution of a scalar Helmholtz equation where the potential (or index) takes two positive values,
one inside a ball of radius $\eps$ and another one outside.  In this paper, we report that the results 
recently obtained in the two dimensional case in \cite{CAPDEBOSCQ-11} can be easily extended to three dimensions. 
In particular, we provide sharp estimates of the size of the scattered field caused by this ball inhomogeneity, 
for any frequencies and any contrast. We also provide a broadband estimate,
that is, a uniform bound for the scattered field for any contrast, and any frequencies outside of a set which 
tends to zero with $\eps$.
\end{abstract}

\maketitle

\section{Introduction}

We consider a scalar field satisfying the Helmholtz equation with frequency $\omega>0$ in $\mathbb{R}^3$.
Given a prescribed incident field $u^i$, a non-singular solution of 
\begin{equation}\label{eq:eq-intro-1}
\Delta u^{i}+\omega^{2}q_{0}u^{i}=0 \mbox{ in } \mathbb{R}^{3},
\end{equation}
we are interested in the solution $u_{\eps}\in H^1_{\mbox{loc}}\left(\mathbb{R}^{3}\right)$ of 
\begin{equation}\label{eq:eq-intro-2}
\Delta u_{\eps}+\omega^{2}q_{\eps}u_{\eps}=0\mbox{ in }\mathbb{R}^{3},
\end{equation}
where, for $|x|>\eps$, $u_{\eps}=u^{i}+u_{\eps}^{s}$, and  $q_{\eps}$ equals $q$ 
inside the inhomogeneity and $q_{0}$ outside. 
We take the inhomogeneity to be a ball of radius $\eps$. The coordinate system is chosen so that the 
inhomogeneity is centered at the origin. In other words
\[
q_{\eps}(r):=\left\{ \begin{array}{ll}
q & \mbox{ if }r<\eps\\
q_{0} & \mbox{ if }r>\eps\end{array}\right.
\]
We assume that both $q_{0}$ and $q$ are real and positive. 
We assume that the scattered field satisfies the classical Silver-M\"uller \cite{MULLER-69} outgoing radiation condition, 
given by
\begin{equation}\label{eq:eq-intro-3}
\frac{\partial}{\partial r}u_{\eps}^{s}-i\omega\sqrt{q_{0}}u_{\eps}^{s}=o\left(\frac{1}{r}\right),
\end{equation}
where, as usual $r:=|x|$.

The purpose of this paper is to provide sharp estimates for the scattered field $u_{\eps}^{s}$, for any contrast $q/q_0$ 
and any frequency $\omega$. The norms we use to describe the scattered field are the following. 
Given any $f \in C^0(\mathbb R^3)$, its restriction to the circle $|x|=R$ can be decomposed in terms of the spherical 
harmonics given by \eqref{eq:def-spherical-harmonics}, in the following way
$$
f(|x|=R) = \sum_{n=0}^\infty\sum_{m=-n}^{n} f_{n,m}(|x|=R) \mathbf{Y}_n^{m}\left(\frac{x}{|x|}\right),
$$
and $f(|x|=R)$ can be measured in terms of the following Sobolev norm
\begin{equation}\label{eq:norm-hs-0}
\left\Vert  f \left(|x|=R\right) \right\Vert _{H^{\sigma}} := 
 \sqrt{\sum_{n=0}^{\infty} \sum_{m=-n}^{n} |f_{n,m}(|x|=R)|^2 (2n+1)^{2\sigma} },
\end{equation}
for any real parameter $\sigma$. By density, this norm can be defined for less regular functions. 
For radius independent estimates, we shall use the norm
\begin{equation}\label{eq:norm-ns-0}
\mathcal{N}^{\sigma} (f) :=  \sqrt{\sum_{n=0}^{\infty}(2n+1)^{2\sigma} \sum_{m=-n}^{n} \sup_{R>0} |f_{n,m}(|x|=R)|^2   }.
\end{equation}
It is easy to see that this norm is finite for a smooth $f$ with bounded radial variations. For a radial function, this is simply the sup 
norm for $f$. Finally, to document the sharpness of our estimates, we will provide lower bounds in terms of the semi-norms
\begin{equation}\label{eq:norm-np-0}
\mathbf{N}_{p,q}^{\sigma} \left( f ,\kappa\right) :=  \sup\limits_{p\leq n \leq q} \sup_{R>0} \kappa^{n+\frac{1}{2} - (n+\frac{1}{2})^{5/6}} (2n+1)^{\sigma} \sqrt{\sum_{m=-n}^{n} |f_{n,m}(|x|=R) |^2} ,
 \end{equation} 
where $q\geq p\geq0$ are integers and $\kappa\geq1$ is a real parameter. These norms are satisfy the following inequality
$$
\left\Vert  f \left(|x|=R\right) \right\Vert _{H^{\sigma}} \leq \mathcal{N}^{\sigma} (f), 
\mbox{ and } \mathbf{N}_{p,\infty}^{\sigma} \left( f,1 \right) \leq \mathcal{N}^{\sigma} (f),
$$
and if for all $R$, $f\left(|x|=R\right)$ only has one non zero spherical harmonic coefficient,
$$
\mathbf{N}_{0,\infty}^{\sigma} \left( f,1 \right) =  \mathcal{N}^{\sigma} (f) = \sup_{R>0} \left\Vert  f \left(|x|=R\right) \right\Vert _{H^{\sigma}}.
$$

They are the natural extension of the norms introduced in \cite{CAPDEBOSCQ-11} for the two dimensional companion problem. When the incident field is a plane wave, 
$$
u^i(\mathbf{x}) = \exp\left(i\sqrt{q_0} \omega \zeta \cdot x\right)
$$
where $\zeta$ is a unit vector in $\mathbb{R}^3$, for all $R>0$, 
$$
\left\Vert  u^{i}\left(|x|=R\right) \right\Vert _{H^{0}} =\mathcal{N}^0(u^i)=\sqrt{\sum_{n=0}^\infty (2n+1) \BJ{n}{\oeps R/\eps}^2}=1,
$$
whereas for any $q\geq p$ and $\kappa\geq1$,
$$
\mathbf{N}^{\sigma}_{p,q}(u^i,\kappa) = \max_{p\leq n\leq q} C(n) \kappa^{n+\frac{1}{2} - (n+\frac{1}{2})^{5/6}} (2n+1)^{\sigma-\frac{1}{3}},
$$
where $2^{\frac{3}{5}}\geq C(n)\geq 1$ for all $p$ and $q$ (see \eqref{eq:bdjn}). 

The motivation from this work comes from imaging.  In electrostatics, the small volume asymptotic expansion for a diametrically bounded conductivity 
inclusion is now well established, and the first order expansion has been shown to be valid for any contrast \cite{NGUYEN-VOGELIUS-09}. 
It is natural to  ask whether such expansions could also hold for non-zero frequencies, even in a simple case.  Another inspiration for this work is recent results 
concerning the so-called cloaking-by-mapping method for the Helmholtz equation. In \cite{KOHN-ONOFREI-VOGELIUS-WEINSTEIN-10, HMNGUYEN-11,LIU-ZHOU-11}, 
the authors show that cloaks can be constructed using lossy layers, and that non-lossy  media could not be made invisible to some particular 
frequencies (the quasi-resonant frequencies). Within the range of non-lossy media, one can ask whether such `cloak busting' frequencies 
are a significant phenomenon, that is, would appear with non-zero probability in any large frequency set, or on the contrary if they 
are contained in a set whose measure tends to zero with $\eps$. 

These questions were considered in two dimensions in \cite{CAPDEBOSCQ-11}. In this paper we show that these results extend, after
some adjustments, to the three dimensional case. 
The proofs presented in this paper are very similar to the ones of the two-dimensional paper, but we believe the results,
more than their derivation, could be of interest to researchers in various areas of mathematics. In numerical analysis 
they could be used as a validation test for broadband Helmholtz solvers, since we provide both upper and lower bounds 
for the scattering data. In the area of small volume expansion for arbitrary geometries, or in the mathematical developments 
related to cloaking,  they provide a 'best case scenario'  which can be used to document the sharpness of more general estimates. 

To make the results of this paper accessible to readers who are not familiar with Bessel functions, the main estimates 
are written in terms of the norms $\|\cdot \|_{H^{\sigma}}, \mathbf{N}_{p,q}^{\sigma}(\cdot,\kappa)$ and $\mathcal{N}^{\sigma}(\cdot)$ introduced above, 
and powers of $2$. Because no unknown constant $C>0$ appears in the results, this paper can be used as a black-box if the reader wishes 
to do so.  Bessel functions do appear in one place, to describe quasi-resonances, but they turn out to be of the same nature as the two-dimensional ones, 
and so we refer to \cite{CAPDEBOSCQ-11} in that case.

The main results of the paper are presented in Section~\ref{sec:main}. The proofs are given in Section~\ref{sec:proofs}.

\section{Main results\label{sec:main}}

To state our results, we introduce the rescaled non-dimensional frequency  $\oeps$, and the contrast factor $\lambda$ given by
\begin{equation}\label{eq:def-oeps-lambda}
 \oeps:= \sqrt{q_0}\omega\eps \, \mbox{ and } \lambda := \sqrt{\frac{q}{q_0}}.
\end{equation}

The following theorem provides our estimates for either small frequencies or for any frequency.

\begin{thm}\label{thm:estim-all} 
For any $R\geq\eps$, when $\max(\lambda,1)\oeps< \frac{1}{2}$ there holds
\begin{equation}\label{eq:omsmall} 
\left\Vert u_{\eps}^{s} \left(|x|=R\right) \right\Vert _{H^{\sigma}}\leq  
2^{4/3} |\lambda -1|\,\oeps  \,\frac{\eps}{R} \left\Vert  u^{i}\left(|x|=\eps\right) \right\Vert _{H^{\sigma-\frac{1}{3}}}.
\end{equation}
If  $p_0\geq0$ is the first integer such that the $(p_0,m)$ spherical harmonic decomposition coefficient of $u_{\eps}^{i} \left(|x|=\eps\right)$ 
is non zero for some $-p_0\leq m\leq p_0$, then \eqref{eq:omsmall} holds for all $\max(\lambda,1)\oeps< p_0+ \frac{1}{2}$.
Furthermore, for any $R\geq\eps$, when $\max(\lambda,1)\oeps< \frac{1}{2}$, the scattered field $u_{\eps}^{s}$ also satisfies
\begin{eqnarray}
\label{eq:corl1} 
&&\left\Vert u_{\eps}^{s} \left(|x|=R\right) \right\Vert _{H^{\sigma}} \\
&&\leq 2^{2/3} |\lambda- 1|\max(\lambda,1) \,\oeps^2  \,\frac{\eps}{R} 
\Big( \left|u^{i}(0)\right|+  2^{2/3}\max(\lambda,1)\oeps \, \mathcal{N}^{\sigma-\frac{1}{3}} \left(u^{i}\right)\Big).\notag
 \end{eqnarray}
When $R\geq \max(\lambda,1) \eps $ there holds
\begin{equation}\label{eq:ffield}
\sup\limits_{\omega>0} \left\Vert u_{\eps}^{s} \left(|x|=R\right) \right\Vert _{H^{\sigma}} 
\leq 
2^{4/3} \,
\max(\lambda,1)\frac{\eps}{R} \,\mathcal{N}^{\sigma} \left(u ^{i} \right).
\end{equation}
\end{thm}
Naturally, the variant of \eqref{eq:omsmall} incorporating the more precise estimate given by \eqref{eq:corl1} for the Fourier coefficient corresponding to $n=0$ 
 also holds by linearity. It is easy to verify that the dependence on $\oeps$ in \eqref{eq:corl1} is optimal by a 
Taylor expansion around $\eps=0$ (or $\omega=0$) for a incident wave $u^i$ with only one (or two) non-zero spherical harmonic coefficients for $n=0$ (and $n=1$). 
Theorem~\ref{thm:estim-all} shows that this estimate is valid up to rescaled frequencies of order $1$ when $\lambda<1$, and of order $1/\lambda$ when $\lambda>1$. 
To prove the optimality of these ranges, we define, for $t<1$,
$$
n_0(t) = \min_{n\geq0} \left\{ n \in \mathbb{N} \mbox{ such that } t^2 \leq 1 - \frac{6}{(n+\frac{1}{2})^{2/3}}\right\},
$$ 
and for $t>1$,
$$
n_1(t) = \min_{n\geq1} \left\{ n \in \mathbb{N} \mbox{ such that } t \leq 1 + \frac{3}{(n+\frac{1}{2})^{2/3}}\right\}.
$$ 
\begin{prop}\label{pro:lowerbounds} When $\lambda < 1$ and $\eps\leq R$,
$$
\sup\limits_{\omega>0} \left\Vert  u_{\eps}^{s} \left(|x|=R \right) \right\Vert _{H^{\sigma}}
\geq \frac{1}{4}\,\frac{\eps}{R}  \mathbf{N}_{n_0(\lambda),\infty}^{\sigma-\frac{1}{6}} \left( u^{i} ,1 \right).
$$
When $\lambda>1$, $R<\eps \lambda$ for any integer $q\geq n_1(\lambda)$,
\begin{equation}\label{eq:blowup}
\sup\limits_{0<\lambda\oeps< 6 q} \left\Vert  u_{\eps}^{s} \left(|x|=R \right) \right\Vert _{H^{\sigma}}
\geq 
\frac{1}{2^{7/2}} \, \mathbf{N}_{n_1(\lambda),q}^{\sigma - \frac{1}{3}} \left( u^{i}, \frac{\eps \lambda}{R} \exp\left(\frac{R}{\eps\lambda}-1\right) \right).
\end{equation}
In fact estimate \eqref{eq:blowup} holds if the supremum is taken in the set
$$
0<\lambda \oeps < q + 1.86 q^{1/3} + 1.04 q^{-1/3}.
$$
Note that $\frac{\eps \lambda}{R} \exp\left(\frac{R}{\eps\lambda}-1\right) >1$ when $R<\eps \lambda$.
\end{prop}

To illustrate the sharp contrast between what these estimates show when $\lambda<1$ and when $1<\lambda$, let us consider the case of a plane wave.
When $\lambda<1$, Theorem~\ref{thm:estim-all} and Proposition~\ref{pro:lowerbounds} show that for any $R\geq\eps$ there holds
$$
\frac{1}{4}\frac{1}{\sqrt{2n_0(\lambda)+1}} \frac{\eps}{R} \leq \sup\limits_{\omega>0} \left\Vert  u_{\eps}^{s} \left(|x|=R \right) \right\Vert _{H^{0}} \leq 2^{4/3} \frac{\eps}{R}
$$  
When $\lambda>1$, Theorem~\ref{thm:estim-all} shows that for any $R \geq\eps$ ,
$$
\sup\limits_{\frac{1}{2}>\lambda \oeps>0} \left\Vert  u_{\eps}^{s} \left(|x|=R \right)\right\Vert _{H^{1/3}} \leq 2^{1/3} \frac{\eps}{R}.
$$
and for any $R\geq \lambda \eps$,
$$
\sup\limits_{\omega>0} \left\Vert  u_{\eps}^{s} \left(|x|=R \right)\right\Vert _{H^{1/3}} \leq 2^{4/3} \frac{\lambda \eps}{R}.
$$
The combination of these two estimates do not provide a bound when $2\lambda \oeps>1$ for the near field $\eps\leq R < \lambda \eps$.  
Proposition~\ref{pro:lowerbounds} provides a lower bound in that case. For any $p\geq n_1(\lambda)$ and any $\sigma$,
$$
\sup\limits_{\oeps<6 p} \left\Vert  u_{\eps}^{s} \left(|x|=R \right)\right\Vert _{H^{\sigma}} 
\geq 2^{-7/2} \left(\frac{\eps \lambda}{R} \exp\left(\frac{R}{\eps\lambda}-1\right)\right)^{p+\frac{1}{2} - (p+\frac{1}{2})^{5/6}} (2p+1)^{\sigma-\frac{2}{3}}.
$$
This lower bound grows geometrically with the upper bound of the interval of frequencies considered. In particular, for any $\lambda>1$, any $R<\lambda\eps$, and  $\sigma\in \mathbb{R}$, 
$$
\sup\limits_{\omega>0} \left\Vert  u_{\eps}^{s} \left(|x|=R \right)\right\Vert _{H^{\sigma}} = +\infty.  
$$

This unbounded behavior of the scattered field is due to the existence of quasi-resonant frequencies, just as in the two-dimensional case. 
To characterize these quasi-resonances, Bessel functions are required. For $t\geq0$, we denote by $\TBH{t}{x}$ 
the Hankel function of the first kind of order $t$. The Bessel functions of the first and second kind of order $t$ are given 
by $\TBJ{t}{x}=\Re\left(\TBH{t}{x}\right)$, and $\TBY{t}{x}=\Im\left(\TBH{t}{x}\right)$. We denote by $\alpha_{\nu,m}$ the m-th positive solution of
$\TBJ{t}{x}=0$. We denote by  $\alpha_{\nu,m}^{(1)}$ the m-th positive solution of $\TBJp{t}{x}=0$. Finally, we write $\beta_{\nu,1}$ the first 
positive solution of $\TBY{t}{x}=0$. 

\begin{defn}\label{def:qr}
For any $t\geq0$, the triplet $(t,x,\lambda)$ is called quasi-resonant if
$$
0<x< \beta_{t,1},
$$
and 
$$
\lambda \TBJp{t}{\lambda x} \TBY{t}{x} = \TBYp{t}{\lambda x} \TBJ{t}{x}.
$$
\end{defn}
The following proposition is proved in \cite{CAPDEBOSCQ-11} in the case when $t$ is an integer, but the proof is unchanged for any $t\geq\frac{1}{2}$.
\begin{prop}\label{pro:Dixon}
For any $t\geq\frac{1}{2}$ and $\lambda > \alpha_{t,1}/\beta_{t,1}$, in every interval
$$
U_{t,k}=\left(\frac{\alpha_{t,k}^{(1)}}{\lambda} ,\frac{\alpha_{t,k}}{\lambda}\right)
\mbox{such that } U_{t,k} \subset \left(\frac{\alpha_{t,1}^{(1)}} {\lambda},\beta_{t,1}\right)
$$
there exists a unique frequency $\omega_{t,k}$ such that the triplet $(t,\omega_{t,k},\lambda)$ is quasi-resonant. 
There are no other quasi-resonances. In particular, no quasi-resonance exists in the interval $(0, \alpha_{t,k}^{(1)}/\lambda)$, or when $\lambda<\alpha_{t,1}/\beta_{t,1}$.
\end{prop}
Since for any $\lambda$ there is only a countable number of quasi-resonant triplets, one could hope that outside security sets around 
the quasi-resonant frequencies, the scattered field could be bounded from above, even in the near field. This means excluding a countable 
union of intervals: a trade-off occurs between how much in the near-field one wishes to go, and how large the set of authorized frequencies is. 
The Theorem below is the result of such a trade-off.
\begin{thm}\label{thm:broadband}
For all $\lambda>0$, all $\epsilon<7^{-3/2}$, $R\geq\eps^{1/3}$, and $\alpha\in (0,1]$, there exists a set $I$ depending on $\eps$, $\lambda$ and $\alpha$ such that
$$
|I|\leq \eps^{1/3} \lneps,
$$
and
$$
\sup\limits_{\sqrt{q_0}\omega \in(0,\infty)\,\setminus \, I} \left\Vert  u_{\eps}^{s}\left(|x|=R\right) \right\Vert _{H^{\sigma}} 
\leq 
\frac{16}{\alpha} \,\frac{\eps^{1/3}}{R} \,\mathcal{N}^{\sigma + 2 + \alpha}\left(u^i\right).
$$
If the contrast $\lambda$ is less than $\eps^{-2/3}$, this holds with $I=\emptyset$.
\end{thm}

\section{Proofs of the main results \label{sec:proofs}}

Altogether, the conditions (\ref{eq:eq-intro-1},\ref{eq:eq-intro-2},\ref{eq:eq-intro-3}) 
imply that the incident field $u^i$, the scattered field $u_\eps^s$ and the transmitted field $u_\eps^{t}=u_\eps$ for $r<\eps$, 
admit series expansions in terms of special functions, namely
\begin{eqnarray}
u^{i}(x)&\sim&
\sum\limits_{n=0}^{\infty} \sum_{m=-n}^{n} a_{n,m} \,\BJ{n}{\sqrt{q_{0}}\omega |x|} \mathbf{Y}_n^m \left( \frac{x}{|x|}\right), \label{eq:uinc}\\
u_\eps^{s}(x)&\sim&
\sum\limits_{n=0}^{\infty} \sum_{m=-n}^{n} a_{n,m} \, \Rn\BH{n}{\sqrt{q_{0}}\omega |x|} \mathbf{Y}_n^m \left( \frac{x}{|x|}\right), \label{eq:usca}\\
u_\eps^{t}(x)&\sim&
\sum\limits_{n=0}^{\infty}  \sum_{m=-n}^{n} a_{n,m} \, \Tn{\oeps}{\lambda} \BJ{n}{\sqrt{q}\omega |x|}\mathbf{Y}_n^m \left( \frac{x}{|x|}\right). \label{eq:utra}
\end{eqnarray}
In the above formulae, $\BJ{n}{x}=\Re(\BH{n}{x})$, and  $x\to\BH{n}{x}$  is the spherical Hankel function of the first kind of order $n$, and $\mathbf{Y}_n^m$ are the spherical harmonics, given in terms of the polar coordinate $\theta\in[0,\pi]$ and $\phi\in[0,2\pi)$ by
\begin{equation}\label{eq:def-spherical-harmonics}
\mathbf{Y}_n^m(\theta,\phi) = \sqrt{\frac{2n+1}{4\pi} \frac{(n-m)!}{(n+m)!}} P_{n}^{m}\left(\cos \theta\right) \exp(i m \phi),
\end{equation}
where $P_{n}^{m}$ is the associated Legendre Polynomial.
 
The reflection and transmission coefficients $r_n$ and $t_n$ are given by the transmission problem on the boundary of the inhomogeneity, that is, at $r=\eps$.
They are the unique solutions of
\begin{eqnarray*}
\Tn{\oeps}{\lambda}  \BJ{n}{\lambda \oeps} 
&=& 
\BJ{n}{\oeps} + \Rn \BH{n}{\oeps}, \\
\lambda  \Tn{\oeps}{\lambda}  \BJp{n}{\lambda\oeps} &=&  
 \BJp{n}{\oeps} +  \Rn  \BHp{n}{\oeps},
\end{eqnarray*}
which are
\begin{equation}\label{eq:def-RN}
\Rn=- \frac{\displaystyle \Re\left(\BHp{n}{\oeps}\BJ{n}{\lambda\oeps}
-\lambda\BJp{n}{\lambda\oeps}\BH{n}{\oeps}
\right)
 }{\displaystyle  
 \BHp{n}{\oeps}\BJ{n}{\lambda\oeps}
-\lambda\BJp{n}{\lambda\oeps}\BH{n}{\oeps}
},
\end{equation}
and, after a simplification using the Wronskian identity satisfied by $\BJ{n}{\cdot}$ and $\BH{n}{\cdot}$,
\begin{equation}\label{eq:def-TN}
\Tn{\oeps}{\lambda}=\frac{i}{\oeps^2}\frac{1}{\displaystyle  
\BHp{n}{\oeps}\BJ{n}{\lambda\oeps}
-\lambda\BJp{n}{\lambda\oeps}\BH{n}{\oeps}
}.
\end{equation}

In \eqref{eq:uinc}, \eqref{eq:usca} and \eqref{eq:utra}, the $\sim$ symbol is an equality if the right-hand-side is replaced by its real part, the fields being real. 
By a common abuse of notations, in what follows we will identify $u^i$ and $u_\eps^s$ with the full complex right-hand-side.

To verify that this is the correct solution, we need to check that $r_n$ and $t_n$ are well defined. The fact that there is a 
unique solution to Problem \eqref{eq:eq-intro-2} satisfying the radiation condition \eqref{eq:eq-intro-3} is well known (see e.g. \cite{MULLER-69}).
\begin{lem} 
$h'_n(\omega_\epsilon) j_n(\lambda \omega_\epsilon) - \lambda j'_n(\lambda \omega_\epsilon) h_n(\omega_\epsilon) $ 
is non zero for all $n\in\mathbb{Z}$, $\oeps>0$ and $\lambda>0$.
\end{lem}
\begin{proof}
Assume, for contradiction, $\BJ{n}{\lambda \oeps} \BHp{n}{\oeps} - \lambda \BJp{n}{\lambda \oeps} \BH{n}{\oeps} = 0$.\\
Then, as $j_n$ and $j'_n$ do not have common zeroes (see e.g. \cite{NIST-10}), either $\BJ{n}{\lambda \oeps}$ is non zero, in which case 
\eqref{eq:eq-intro-2}-\eqref{eq:eq-intro-3} has the following solution
\begin{equation}\label{eq:uniq-1}
U_\epsilon(r) := \left\{
\begin{array}{ll}
\frac{\BH{n}{\oeps}}{\BJ{n}{\lambda \oeps}} \BJ{n}{\lambda \oeps \frac{r}{\epsilon}} & \text{if } r < \epsilon,\\
\BH{n}{\oeps \frac{r}{\epsilon}} & \text{if } r > \epsilon,
\end{array}\right.
\end{equation}
Or $\BJp{n}{\lambda \oeps}$ is non zero, and \eqref{eq:eq-intro-2}-\eqref{eq:eq-intro-3}  has the following solution 
\begin{equation}\label{eq:uniq-2}
U_\epsilon(r) := \left\{
\begin{array}{ll}
\frac{\BHp{n}{\oeps}}{\lambda \BJp{n}{\lambda \oeps}} \BJ{n}{\lambda \oeps \frac{r}{\epsilon}} & \text{if } r < \epsilon,\\
\BH{n}{\omega_\epsilon \frac{r}{\epsilon}} & \text{if } r > \epsilon.
\end{array}\right.
\end{equation}
Both \eqref{eq:uniq-1} and \eqref{eq:uniq-2} would be solution of problem \eqref{eq:eq-intro-2}-\eqref{eq:eq-intro-3} without an incident wave. There is of course another 
solution  to that problem, $U_\eps\equiv0$. Since Problem \eqref{eq:eq-intro-2}-\eqref{eq:eq-intro-3} is well posed, see \cite{MULLER-69}, 
we have a contradiction.
\end{proof}

We chose the three (semi-)norms $\left\Vert  \cdot \right\Vert _{H^{\sigma}}$, $\mathcal{N}^{\sigma}$  and $\mathbf{N}_{p,q}^{\sigma}(\cdot,\kappa)$ because they are 
compatible with expansions \eqref{eq:uinc}, \eqref{eq:usca} and \eqref{eq:utra}. In particular, for any $R>0$, we have
\begin{equation}\label{eq:norm-us-hs}
\left\Vert  u_{\eps}^{s}\left(|x|=R\right) \right\Vert _{H^{\sigma}} :=   
\left(\sum_{n=0}^{\infty} \sum_{m=-n}^{n} \left|a_{n,m}\right|^2
(2n+1)^{2\sigma}  \left| \Rn \BH{n}{\oeps R/\eps}\right|^{2}\right)^{\frac{1}{2}},
\end{equation} 
and 
\begin{equation}\label{eq:norm-ui-hs}
\left\Vert  u^{i}\left(|x|=\eps\right) \right\Vert _{H^{\sigma}} :=   
\left(\sum_{n=0}^{\infty}\sum_{m=-n}^{n} \left| a_{n,m}\right|^2
(2n+1)^{2\sigma}  \left|\BJ{n}{\oeps}\right|^{2}\right)^{\frac{1}{2}},
\end{equation} 
Whereas the other norms are 
\begin{equation}\label{eq:norm-ui-nc}
\mathcal{N}^{\sigma} \left( u ^{i} \right) :=  \left(\sum_{n > 0 }\sum_{m=-n}^{n}\left|a_{n,m}\right|^{2} \sup_{x>0}\left|\BJ{n}{x}\right|^2(2n+1)^{2\sigma}\right)^{\frac{1}{2}},
\end{equation} 
and for any $q\geq p\geq0$ and $\kappa>1$,
\begin{equation}\label{eq:norm-ui-n0}
\mathbf{N}_{p,q}^{\sigma} \left( u^{i},\kappa \right) :=   \sup\limits_{ p\leq n \leq q} \kappa^{n+\frac{1}{2} - (n+\frac{1}{2})^{5/6}} (2n+1)^{\sigma} \sup_{x>0}\left|\BJ{n}{x} \right|
                                     \sqrt{\sum_{m=-n}^{n} \left|a_{n,m}\right|^2} .
\end{equation} 

\subsection{ Proof of Theorem~\ref{thm:estim-all}}
The two dimensional results found in \cite{CAPDEBOSCQ-11} are easily translated into three dimensional ones for the following reason.
\begin{prop}\label{pro:2d3d}
Let $\TRn$ be the reflection coefficient associated to problem \eqref{eq:eq-intro-2} posed in a disk of radius $\eps$ in dimension 2, 
with the appropriate out-going radiation condition. This reflection coefficient is defined (by the same formula) when $n$ is an 
arbitrary positive real number, and $\oeps$ is replaced by any real $x>0$. Then, for any $n\geq0$ and any $x>0$ there holds
$$
r_{n}(x,\lambda) = R_{n+\frac{1}{2}}(x,\lambda).
$$
\end{prop}
\begin{nota}
From this point onwards, we use the short-hand $\nu$ to represent the number $n+\frac{1}{2}$. 
\end{nota}

\begin{proof}
For $ n \geq 0$ we introduce 
$$
g_{\nu}(x) = \frac{x}{\nu} \frac{\TBJp{\nu}{x}}{\TBJ{\nu}{x}}, \mbox{ and } k_{\nu}(x) = -\frac{x}{\nu} \frac{\TBYp{\nu}{x}}{\TBY{\nu}{x}}. 
$$
Then, the reflection function $R_{\nu}$ introduced in two dimensional problem considered in \cite{CAPDEBOSCQ-11} is
$$
R_{\nu}(x,\lambda)=-\frac{g_{\nu}(\lambda x) - g_{\nu}(x) }{ g_{\nu}(\lambda x) - g_{\nu}(x) +{\bf i} \tan \theta_{\nu}(x) \left(  g_{\nu}(\lambda x) + k_{\nu}(x) \right) }.
$$
where ${\mathbf i}^2=-1$. For any $\nu>0$ and any $x$ such that $\TBJ{\nu}{x}\neq0$, we write
$$
\tan \left( \theta_{\nu}(x)\right) = \frac{\TBY{\nu}{x}}{\TBJ{\nu}{x}}.
$$
The properties of the function $R_{t}$ were studied in \cite{CAPDEBOSCQ-11} for any $t\geq1$. When $t$ is an integer, $R_{t}$ is the reflection coefficient 
associated to problem \eqref{eq:eq-intro-2} in dimension $2$. Note that the identities
\begin{equation}\label{eq:2d3dbj}
\BJ{t-\frac{1}{2}}{x} = \sqrt{\frac{\pi}{2x}} \TBJ{t}{x} \mbox{ and } \BY{t-\frac{1}{2}}{x} = \sqrt{\frac{\pi}{2x}} \TBY{t}{x},
\end{equation}
valid for any $t\in\mathbb{R}$ yield that
$$
\tan \left( \theta_{\nu}(x)\right) = \frac{\BY{n}{x}}{\BJ{n}{x}}.
$$
We introduce the notations
\begin{eqnarray*}
u_n & = &  x \BJp{n}{x} \BJ{n}{\lambda x} -\lambda x \BJp{n}{\lambda x} \BJ{n}{x} \\
v_n & = &  x \BYp{n}{x} \BJ{n}{\lambda x} -\lambda x \BJp{n}{\lambda x} \BY{n}{x},
\end{eqnarray*}
Then, the reflection coefficient corresponding to \eqref{eq:eq-intro-2} is given by
$$
r_n(x,\lambda) =\frac{-u_n}{u_n+ {\mathbf i} v_n}.
$$
If we introduce for any $t\in\mathbb{R}$, outside the zeroes of $\BJ{t-\frac{1}{2}}{x}$ and $\BY{t-\frac{1}{2}}{x}$, 
$$ 
f_t(x) = \frac{x}{t} \frac{\BJp{t-\frac{1}{2}}{x}}{\BJ{t-\frac{1}{2}}{x}}, \mbox{ and } s_t(x) = -\frac{x}{t} \frac{\BYp{t-\frac{1}{2}}{x}}{\BY{t-\frac{1}{2}}{x}},
$$
we can rewrite, when $\lambda x$ is not a zero of $\BJ{n}{x}$, $\BJ{n}{\lambda x}$ or  $\BY{n}{x}$,
\begin{eqnarray*}
u_n &=&  -n \BJ{n}{x}\BJ{n}{\lambda x} ( f_{\nu}(\lambda x) - f_{\nu}(x)), \\
v_n &=&  -n \tan \theta_{\nu}(x) \BJ{n}{x}\BJ{n}{\lambda x} ( f_{\nu}(\lambda x) + s_{\nu}(x)).
\end{eqnarray*}
From \eqref{eq:2d3dbj} it follows that for any $t>0$,
$$
f_t(x) = g_t(x) - \frac{1}{2t} \mbox{ and } s_t(x) = k_t(x) + \frac{1}{2t},
$$
where these functions are defined. This in turns implies that
\begin{eqnarray*}
-\frac{u_n}{u_n+{\mathbf i} v_n} &=& -\frac{f_{\nu}(\lambda x) - f_{\nu}(x) }{ f_{\nu}(\lambda x) - f_{\nu}(x) +{\mathbf i} \tan \theta_{\nu}(x) 
\left(  f_{\nu}(\lambda x) + s_{\nu}(x) \right) } \\
&=& -\frac{g_{\nu}(\lambda x) - g_{\nu}(x) }{ g_{\nu}(\lambda x) - g_{\nu}(x) +{\mathbf i} \tan \theta_{\nu}(x) \left(  g_{\nu}(\lambda x) + k_{\nu}(x) \right) } \\
&=& R_{\nu}(x,\lambda),
\end{eqnarray*}
and this identity extends to the zeroes of $\BJ{n}{x}$, $\BJ{n}{\lambda x}$ or  $\BY{n}{x}$ by continuity.
\end{proof}
The following Lemma then follows mostly from \cite{CAPDEBOSCQ-11}. 
\begin{lem}\label{lem:7.1}
For any $\lambda>0$ and $n\geq0$,
\begin{itemize}
\item For $x\in(0,\min(\alpha_{\nu,1}^{(1)}/\lambda,\beta_{\nu,1}))$ 
there holds
\begin{equation}\label{eq:uniform-lowom}
 \left|r_n(x,\lambda)\BH{n}{x}\right| \leq 2^{\frac{4}{3}} \BJ{n}{x}.
\end{equation}
\item For $x\in(0, \min(1/\lambda,1) \nu)$ we have
\begin{equation}\label{eq:expan-perturb}
 \left|r_n(x,\lambda)\BH{n}{x}\right| \leq 2^{\frac{4}{3}} |\lambda -1|\frac{x}{(2\nu)^{1/3}}\BJ{n}{x}.
\end{equation}
\item Furthermore, when $0<\max(\lambda,1) x<\frac{1}{2}$
\begin{equation}\label{eq:expan-perturb2}
 \left|r_0(x,\lambda)\BH{0}{x}\right| \leq 2^{\frac{2}{3}} |\lambda -1|\max(\lambda,1) x^2,
\end{equation}
and when $0<\max(\lambda,1) x<\nu$ and $n\geq1$,
\begin{equation}\label{eq:expan-perturb3}
 \left|r_n(x,\lambda)\BH{n}{x}\right| \leq 2^{\frac{4}{3}} |\lambda -1|\max(\lambda,1)^2 \frac{x^3}{(2\nu)^{1/3}}\BJ{n}{1}.
\end{equation}
\item When $\lambda < 1$ and $n\geq n_0(\lambda)$, we have
\begin{equation}\label{eq:lowl-contrex}
 \left|r_n(\nu,\lambda)\BH{n}{\nu}\right| > \frac{1}{2} \BJ{n}{\nu}.
\end{equation}
\item When $\lambda>n_1(\lambda)$ there exists $x_n < \alpha_{\nu,1}/\lambda$ such that
\begin{equation}\label{eq:quasi-res-contrex}
\left|r_n(x_n,\lambda)\right| =1.
\end{equation}
\end{itemize}
\end{lem}
\begin{proof}
The existence of $x_n$ satisfying \eqref{eq:quasi-res-contrex} follows from Proposition~\ref{pro:Dixon}.
Thanks to Proposition~\ref{pro:2d3d}, and because for any $n\geq0$, we have
$$
\frac{\BJ{n}{x}}{\BH{n}{x}} = \frac{\TBJ{\nu}{x}}{\TBH{\nu}{x}},
$$
the inequalities \eqref{eq:uniform-lowom}, \eqref{eq:expan-perturb} and \eqref{eq:lowl-contrex}  are proved  when $n\geq1$ in \cite{CAPDEBOSCQ-11},  Lemma 7.1.
We will now check that \eqref{eq:uniform-lowom} holds when $n=0$. 
From \eqref{eq:def-RN} there holds
 $$
 |r_0(x,\lambda)|\leq 1. 
 $$ 
Since 
$$
\TBJ{\frac{1}{2}}{x} = -\tan(x) \TBY{\frac{1}{2}}{x}, 
$$
we have for all $x\in (3/5,\pi/2)$
\begin{equation}\label{eq:bd-d1}
 \left|r_0(x,\lambda)\BH{0}{x}\right| \leq \BJ{n}{x} \left( 1 + \frac{1}{\tan(3/5)}\right) \leq \frac{5}{2} \BJ{n}{x}.
\end{equation}
 The zero  $\alpha_{\frac{1}{2},1}^{(1)}$ is given the first positive solution of $\tan(x)=2x$. 
It satisfies $1<\alpha_{\frac{1}{2},1}^{(1)}< \beta_{\frac{1}{2},1} =\frac{\pi}{2}$.  We can thus consider only 
$0<x<\min(\frac{3}{5}, \alpha_{\frac{1}{2},1}^{(1)} \lambda^{-1})$.

We  have from Proposition~\ref{pro:2d3d}, for $n=0,1,2$
\begin{eqnarray*}
&& r_{n}(x,\lambda)\BH{n}{x}\\
&=&  - \BJ{n}{x}\frac{\left(g_{n+\frac{1}{2}}(\lambda x) - g_{n+\frac{1}{2}}(x)\right)(1+ {\mathbf i} \tan \theta_{n+\frac{1}{2}}(x)) }
{ g_{n+\frac{1}{2}}(\lambda x) - g_{n+\frac{1}{2}}(x) + {\mathbf i}\tan \theta_{n+\frac{1}{2}}(x) \left(  g_{n+\frac{1}{2}}(\lambda x) + k_{n+\frac{1}{2}}(x)\right) }\\
&=&  - \BJ{n}{x} s_n(x),
\end{eqnarray*}
and using the Wronskian identity satisfied by $\TBJ{n+\frac{1}{2}}{x}$, $\TBY{n+\frac{1}{2}}{x}$ and the recurrence relations satisfied by Bessel functions, 
we obtain (see \cite{CAPDEBOSCQ-11} 
for details) that
$$
s_n(x) = \frac{u_{n} (x)}{u_{n}(x)+ {\mathbf i} },
$$
with
$$
u_n(x) = (2n+1)\frac{\pi}{4}\TBH{n+\frac{1}{2}}{x} \TBJ{n+\frac{1}{2}}{x} \left(  g_{n+\frac{1}{2}}(x) -  g_{n+\frac{1}{2}}(\lambda x)\right).
$$
Note that $g_{n+\frac{1}{2}}(x)$, $\TBJ{n+\frac{1}{2}}{x}$ and $|\TBH{n+\frac{1}{2}}{x}|$ have simple analytic formulae. For example
$$ 
g_{\frac{1}{2}}(x)=2\frac{x}{\tan x} -1,\quad \TBJ{\frac{1}{2}}{x}= \sqrt{\frac{2}{\pi x}} \sin(x),\quad\mbox{ and } |\TBH{\frac{1}{2}}{x}|=\sqrt{\frac{2}{\pi x}}.
$$
In particular, it is easy to show that $g_{n+\frac{1}{2}}$ is decreasing on $(0,\alpha_{n+\frac{1}{2},1}^{(1)})$. 

When $n=0$, we have the following bound for  all $0<x<\min(3/5, \alpha_{\frac{1}{2},1}^{(1)} \lambda^{-1})$.
$$
|u_0(x)| \leq \max_{0<x<\frac{3}{5}}|\frac{\pi}{4}\TBH{\frac{1}{2}}{x}\TBJ{\frac{1}{2}}{x}| 
\max_{0\leq x < \alpha_{\frac{1}{2},1}^{(1)}} |g_{\frac{1}{2}}(x)| \leq \frac{1}{2}.
$$
As a consequence, 
$$
\left|s_0(x)\right| \leq \frac{|u_0(x)|}{|1-|u_0(x)||} \leq 1.
$$
Which concludes the proof of \eqref{eq:uniform-lowom} when $n=0$. 

For any $n=0,1,2$, and $\max(\lambda,1)x<n+\frac{1}{2}$, we have
\begin{eqnarray*}
|u_n(x)| &\leq& \max_{0<x<n+\frac{1}{2}}|\frac{(2n+1)\pi}{2}\TBH{n+\frac{1}{2}}{x}\TBJ{n+\frac{1}{2}}{x}| \max_{0\leq x < n+\frac{1}{2}} |g_{n+\frac{1}{2}}(x)| \\
&=& \frac{2n+1}{2} M_n <1, 
\end{eqnarray*}
where $M_n = \max_{0<x<n+\frac{1}{2}}|(2n+1)\frac{\pi}{4}\TBH{n+\frac{1}{2}}{x}\TBJ{n+\frac{1}{2}}{x}|$ for $n=0,1,2$. Therefore
\begin{equation}\label{eq:intera-1}
\left|s_n(x)\right| \leq \frac{|u_n(x)|}{|1-|u_n(x)||}\leq \frac{2}{2-(2n+1)M_n} |u_n(x)|.
\end{equation}
To proceed, note that $\frac{d^2}{dx^2}g_{\frac{1}{2}}$ is negative and decreasing, and for $x<n+\frac{1}{2}$,
$$
\left|\frac{d^2}{dx^2}g_{n+\frac{1}{2}}(x)\right|  \leq -\kappa_n =\frac{d^2}{dx^2}g_{n+\frac{1}{2}}\left(\frac{n+1}{2}\right).
$$
Thus a Taylor expansion shows that for $0<x<n+\frac{1}{2}$,
\begin{equation}\label{eq:intera-2}
|u_n(x)| \leq M_n\kappa_n|\lambda -1|\max(\lambda,1) x^2 .
\end{equation} 
Inserting \eqref{eq:intera-2} in \eqref{eq:intera-1} together with the values of $M_0$ and $\kappa_0$,  we obtain 
$$
\left|s_0(x)\right|  \leq  2^{\frac{2}{3}}|\lambda -1|\max(\lambda,1) x^2 
$$
for all $0< 2\max(\lambda,1)x \leq 1$, which implies \eqref{eq:expan-perturb}, and \eqref{eq:expan-perturb2} since $|\BJ{0}{x}|\leq1$. 

Inserting \eqref{eq:intera-2} in \eqref{eq:intera-1} together with the values of $M_n$ and $\kappa_n$ for $i=n$ and $n=2$,  
we obtain for $x\leq n+\frac{1}{2}$
\begin{equation}\label{eq:intera-3}
\left|s_n(x)\right| \leq \sqrt{3}|\lambda -1| \max(\lambda,1) x^2 \left(n+\frac{1}{2}\right)^{-1/3}.
\end{equation}
For all $n\geq 0$, it is known that \cite{PARIS-84} for $0<x\leq y<\alpha_{\nu,1}^{(1)}$,
\begin{equation}\label{eq:paris-ineq}
\BJ{n}{x} \leq \frac{x^n}{y^n} \BJ{n}{y} \exp \left( \frac{y^2-x^2}{2n+5}\right).
\end{equation}
In particular, for all $x \leq 1$, we have 
\begin{equation}\label{eq:intera-4}
\BJ{1}{x} \leq  x \exp\left(\frac{1}{7}\right)\BJ{1}{1},\mbox{ and } \BJ{n}{x} \leq  x^2 \exp\left(\frac{1}{9}\right) \BJ{n}{1} \mbox{ for } n\geq 2.
\end{equation}
Combining \eqref{eq:intera-3}, \eqref{eq:intera-4} and \eqref{eq:expan-perturb} we obtain \eqref{eq:expan-perturb3}.
\end{proof}

We may now conclude the proof of Theorem~\ref{thm:estim-all}.
\begin{proof}[Proof of Theorem~\ref{thm:estim-all}]
For convenience we write $A:=\left\Vert  u_{\eps}^{s}\left(|x|=R\right) \right\Vert_{H^{\sigma}}$. 
Formula~\eqref{eq:norm-us-hs} shows that
$$
A^2 =   
\sum_{n=p_0}^{\infty} \sum_{m=-n}^{n} \left|a_{n,m}\,(2\nu)^{\sigma} \Rn\BH{n}{\oeps}\right|^2  \left|\frac{\BH{n}{ \oeps R/\eps }}{\BH{n}{\oeps}}\right|^{2} 
$$
Note that $x\left|\BH{n}{x}\right|$ is decreasing (see e.g. \cite{WATSON-25} $\S$ 13.74), therefore for all $R\geq \eps$,
\begin{equation}\label{eq:pf21-1}
A^2  \leq   \frac{\eps^2}{R^2}
\sum_{n=p_0}^{\infty} \sum_{m=-n}^{n} \left|a_{n,m}\, (2\nu)^{\sigma} \Rn\BH{n}{\oeps}\right|^2  .
\end{equation}
Estimate \eqref{eq:expan-perturb} in Lemma~\ref{lem:7.1} shows that when $\max(\lambda,1)\oeps \leq p_0+\frac{1}{2} \leq \nu$.
$$
(2\nu)^{\sigma} \left|\Rn\BH{n}{\oeps}\right|  \leq 2^{4/3} |\lambda-1|\oeps (2\nu)^{\sigma - \frac{1}{3}}\left|\BJ{n}{\oeps}\right|.
$$
Inserting this bound in \eqref{eq:pf21-1} we obtain
$$
A \leq    2^{4/3} |\lambda-1|\oeps \frac{\eps}{R}
\left\Vert  u^{i}\left(|x|=\eps\right) \right\Vert_{H^{\sigma-\frac{1}{3}}},
$$
which is estimate \eqref{eq:omsmall}.  Using  \eqref{eq:expan-perturb3} instead, we obtain for $n\geq1$ and $\oeps\leq1$,
$$
(2\nu)^{2\sigma} \left|\Rn\BH{n}{\oeps}\right|  \leq 2^{4/3} |\lambda-1|\max(\lambda,1)^2 \oeps^3 \nu^{2\sigma - \frac{1}{3}}\left|\BJ{n}{1}\right|^2.
$$
Inserting this bound in \eqref{eq:pf21-1} together with \eqref{eq:expan-perturb2} we obtain
$$
A \leq  2^{2/3} \oeps^2 |\lambda-1|\max(\lambda,1)\frac{\eps}{R} \left( |a_0| +
2^{2/3}  \oeps \max(\lambda,1)\mathcal{N}^{\sigma-\frac{1}{3}}(u^{i})\right),
$$
which proves \eqref{eq:corl1} since $|a_0|=|u^i(0)|$.
Let us now turn to \eqref{eq:ffield}. Formula~\eqref{eq:norm-us-hs} shows that
$$
A^2=\sum_{n=p_0}^{\infty} \sum_{m=-n}^{n} \left|a_{n,m} \, (2\nu)^{\sigma} \Rn\BH{n}{\max(\lambda,1)\oeps}\right|^2  \left|\frac{\BH{n}{ \oeps R/\eps }}{\BH{n}{\max(\lambda,1)\oeps}}\right|^{2}.
$$
When $R\geq \max(\lambda,1) \eps$, since  $x\left|\BH{n}{x}\right|$  is decreasing
\begin{equation}\label{eq:pf21-2}
A^2 \leq    \left(\max(\lambda,1)\frac{\eps}{R}\right)^2
\sum_{n=p_0}^{\infty} \sum_{m=-n}^{n} \left| a_{n,m}\,(2\nu)^{\sigma} \Rn \BH{n}{\max(\lambda,1) \oeps}\right|^2.  
\end{equation}
We define two sets of indices, 
$$
I:=\left\{n \in \mathbb{N} \mbox{ such that } \oeps \leq \min(\alpha_{\nu,1}^{(1)}\lambda^{-1}, \beta_{\nu,1}) \right\}
$$
and $J=\mathbb{N}\setminus I$.  
Since $\left|\BH{n}{x}\right|$ is decreasing, estimate \eqref{eq:uniform-lowom} shows that for all $n\in I$, we have
\begin{equation*}
\left|\Rn \BH{n}{\max(\lambda,1)\oeps}\right| \leq \left|\Rn\BH{n}{\oeps}\right| 
\leq 2^{4/3}   \sup_{x>0} \left|\BJ{n}{x}\right|. 
\end{equation*}
On the other hand, $\nu<\alpha_{\nu,1}^{(1)}<\beta_{\nu,1}$, therefore when $n \in J$, $\max(\lambda,1)\oeps \geq \alpha_{\nu,1}^{(1)}$, and it is known 
(see \cite{WATSON-25} $\S$ 13.74) that when $x>n$, $\sqrt{x^2-n^2}\left|\TBH{n}{x}\right|^2$ is an increasing function of $x$ with limit $2/\pi$. Furthermore, it is 
also known (see \cite{WATSON-25} $\S$ 15.3) that for all $n\geq0$
$$
\alpha_{\nu,1}^{(1)} > \nu+\frac{4}{5}\nu^{1/3}.
$$
Therefore, since $|\Rn|\leq1$, we have
\begin{eqnarray*}
\left|\Rn \BH{n}{\max(\lambda,1)\oeps}\right|^2 
&\leq& \left|\BH{n}{\max(\lambda,1)\oeps}\right|^2, \\
&= &\frac{2}{\pi \max(\lambda,1)\oeps} \left|\TBH{\nu}{\max(\lambda,1)\oeps}\right|^2 .
\end{eqnarray*}
On the other hand,
$$
\frac{2}{\pi \max(\lambda,1)\oeps} \left|\TBH{\nu}{\max(\lambda,1)\oeps}\right|^2 
\leq \frac{2}{\pi \alpha_{\nu,1}^{(1)}} \frac{2}{\pi\sqrt{(\alpha_{\nu,1}^{(1)})^2 - \nu^2}} 
\leq \frac{2}{\pi \alpha_{\nu,1}^{(1)}} \frac{3}{5} \frac{1}{\nu^{2/3}}.
$$
As it is known (see \cite{LANDAU-00}) that for all $n\geq0$
\begin{equation}\label{eq:bound-landau}
0.539<\nu^{1/3}\sup_{x>0}\TBJ{\nu}{x}=\nu^{1/3} \TBJ{\nu}{\alpha_{\nu,1}^{(1)}}<0.675,
\end{equation}
this shows that
$$
\frac{2}{\pi \max(\lambda,1)\oeps} \left|\TBH{\nu}{\max(\lambda,1)\oeps}\right|^2 
\leq \frac{2}{\pi \alpha_{\nu,1}^{(1)}} \TBJ{\nu}{\alpha_{\nu,1}^{(1)}}^2
=  \left|\BJ{n}{\alpha_{n,1}^{(1)}}\right|^2.
$$
We have obtained that, for $n\in I\cup J$,
\begin{equation}\label{eq:uniformjj}
\left|\Rn \BH{n}{\max(\lambda,1)\oeps}\right| \leq   2^{4/3} \sup_{x>0} \left|\BJ{n}{x}\right|. 
\end{equation}
Inserting this bound in \eqref{eq:pf21-2}, we obtain \eqref{eq:ffield}.
\end{proof}

\subsection{Proof of Proposition~\ref{pro:lowerbounds}}
To prove Proposition~\ref{pro:lowerbounds}, we shall use the following intermediate result, related to Bessel functions.
\begin{lem}\label{lem:boundhn}
For any $n\geq0$, there holds
\begin{equation}\label{eq:bdjn}
\frac{1}{2\sin\frac{1}{2}} \sup_{x>0} |\BJ{n}{x}| >\frac{1}{(2\nu)^{5/6}} > 0.663 \sup_{x>0} |\BJ{n}{x}|
\end{equation}
For any $n>n_0(\lambda)$ there holds
\begin{equation}\label{eq:bdb1}
\left|\BH{n}{\nu\frac{R}{\eps}}\right| > \frac{\eps}{R}\frac{0.58}{\nu^{1/6}} \frac{\left|\BH{n}{\nu}\right|}{\BJ{n}{\nu}}  \sup_{x>0} |\BJ{n}{x}|. 
\end{equation}
For any $\lambda>1$, $n\geq1$, and $\frac{R}{\lambda\eps}\leq 1$, there holds
\begin{equation}\label{eq:bdb2}
\left|\BH{n}{\frac{R}{\eps \lambda} \alpha_{\nu,1}}\right| > 2^{-7/2}\frac{1}{(2\nu)^{1/3}}  \left( \frac{\eps \lambda}{R} \exp\left(\frac{R}{\eps\lambda}-1\right) \right)^{\nu-\nu^{5/6}}\sup_{x>0} |\BJ{n}{x}|.
\end{equation}
\end{lem}
We prove this lemma below. We can now conclude the proof of Proposition~\ref{pro:lowerbounds}.
\begin{proof}[Proof of Proposition~\ref{pro:lowerbounds}]
Let us start with the case $\lambda<1$. Starting as before from formula~\eqref{eq:norm-us-hs}, we have
\begin{eqnarray*}
&&\sup_{\oeps>0}
\left\Vert  u_{\eps}^{s}\left(|x|=R\right) \right\Vert_{H^{\sigma}} \\
&\geq&   
\sup_{\oeps>0} \sup_{n\geq n_0(\lambda)} \sum_{m=-n}^{n} \sqrt{\left|a_{n,m}\right|^2
 \left|\Rn\BH{n}{\oeps}\right|^2}  (2\nu)^{\sigma}  \left|\frac{\BH{n}{ \oeps R/\eps }}{\BH{n}{\oeps}}\right|  \\
 &\geq& \sup_{n\geq n_0(\lambda)} \sqrt{\left|a_{n,m}\right|^2
 \left|r_n(\nu,\lambda)\BH{n}{\nu}\right|^2 } (2\nu)^{\sigma}
  \left|\frac{\BH{n}{ \nu R/\eps }}{\BH{n}{\nu}}\right|.
\end{eqnarray*}
Using now the bounds \eqref{eq:bdb1} in Lemma~\ref{lem:boundhn} and \eqref{eq:lowl-contrex} in Lemma~\ref{lem:7.1}, we have
$$
(2\nu)^{\sigma}
 \left|r_n(\nu,\lambda)\BH{n}{\nu}\right|  \left|\frac{\BH{n}{ \nu R/\eps }}{\BH{n}{\nu}}\right| 
> \frac{1}{4} \,\frac{\eps}{R} (2\nu)^{\sigma-\frac{1}{6}} \sup_{x>0} |\BJ{n}{x}|,
$$
Therefore
$$
\sup_{\oeps>0}
\left\Vert  u_{\eps}^{s}\left(|x|=R\right) \right\Vert_{H^{\sigma}} \geq\frac{1}{4} \frac{\eps}{R} \sup_{n\geq n_0(\lambda)} \sqrt{\left|a_{n,m}\right|^2
  \sup_{R>0}\left|\BJ{n}{R}\right|^2 } (2\nu)^{\sigma-\frac{1}{6}},
$$
as claimed. We now turn to the case $\lambda>1$. Note that for all $q\geq1$, using the upper bound given in \eqref{eq:cn} we see that $\alpha_{q+\frac{1}{2},1}<6 q$. 
Therefore, choosing for each $q$ the frequency $x_q$ given by Lemma~\ref{lem:7.1}, we have
\begin{eqnarray*}
\sup_{0<\lambda\oeps<6q}
\left\Vert  u_{\eps}^{s}\left(|x|=R\right) \right\Vert_{H^{\sigma}} &\geq&   
\sup_{1< n< q}  \sqrt{\sum_{m=-n}^{n}\left|a_{n,m}\right|^2}
 \left|\BH{n}{x_q \frac{R}{\eps}}\right|  (2\nu)^{\sigma} \\
&\geq&   
\sup_{1< n< q}  \sqrt{\sum_{m=-n}^{n}\left|a_{n,m}\right|^2}
 \left|\BH{n}{\frac{\alpha_{q+\frac{1}{2},1}}{\lambda} \frac{R}{\eps}}\right|  (2\nu)^{\sigma}.
\end{eqnarray*}
The conclusion follows from estimate \eqref{eq:bdb2}.
\end{proof}

\begin{proof}[Proof of Lemma~\ref{lem:boundhn}]
To estimate the maximum of $|\BJ{n}{x}|$, we proceed as follows. 
Note that the maximal value occurs at the first positive solution of $\BJp{n}{x}=0$, which we will denote $\gamma_{\nu,1}$. We compute that
$$
\BJp{n}{x} = \sqrt{\frac{\pi}{2}}\frac{\TBJ{\nu}{x}}{x^{3/2}}\left( \frac{x \TBJp{\nu}{x}}{\TBJ{\nu}{x}} -\frac{1}{2}\right).
$$
It is known (see e.g. \cite{CAPDEBOSCQ-11} Proposition A.1) that for $x<\nu$ and $\nu>1$, 
$$
x \frac{\TBJp{\nu}{x}}{\TBJ{\nu}{x}} > \frac{\nu^{2/3}}{\sqrt{2}} >\frac{1}{2},
$$
Therefore $\gamma_{\nu,1}>\nu$. 
This implies that
$$
\sqrt{\frac{\pi}{2 n}}\sup_{x>0} \TBJ{\nu}{x} = \sqrt{\frac{\pi}{2 n}} \TBJ{n}{\alpha_{\nu,1}^{(1)}} > \sup_{x>0} |\BJ{n}{x}|.
$$
Next, note that from \eqref{eq:bound-landau}, and the bound $\nu< \alpha_{\nu,1}^{(1)}$ we have
$$
\frac{1}{(2\nu)^{5/6}} > \frac{\nu^{1/3} \TBJ{\nu}{\alpha_{\nu,1}^{(1)}}}{0.675 (2\nu)^{5/6}} 
> \frac{1}{0.675 2^{1/3} \sqrt{\pi}}  \sqrt{\frac{\pi}{2 n}}\sup_{x>0} \TBJ{\nu}{x}  > 0.663  \sup_{x>0} |\BJ{n}{x}| .
$$
On the other hand,  $\nu \to \nu^{1/3}\TBJ{\nu}{\nu}$ is an increasing function (see \cite{WATSON-25} $\S$ 15.8), therefore
$$
2 \sin \frac{1}{2} = \sqrt{\pi} \TBJ{\frac{1}{2}}{\frac{1}{2}} \leq \sqrt{\pi} \TBJ{\nu}{\nu} (2\nu)^{1/3} = (2\nu)^{5/6} |\BJ{n}{\nu}|< (2\nu)^{5/6} \sup_{x>0} |\BJ{n}{x}|.
$$
It is well known (see e.g. \cite{NIST-10}) that for all $x>0$, $|\BH{\nu}{x}|>\frac{1}{x}$.
It is also known (see \cite{WATSON-25} $\S$ 15.8) that 
$\nu\to-\frac{\BY{\nu}{\nu}}{\BJ{\nu}{\nu}}$ is a decreasing function. Note if $n>n_0(\lambda)$ then $\nu>6^{3/2}$. Therefore
$$
|\BH{n}{\nu}| < \BJ{n}{\nu}\sqrt{1+\frac{\BY{10}{10}^2}{\BJ{10}{10}^2}} < 2.01 \BJ{n}{\nu}.
$$
Combining these two bounds we obtain that
\begin{equation}\label{eq:estim01}
\left|\BH{n}{\nu\frac{R}{\eps}}\right| > \frac{\eps}{R}\frac{1}{2.01\,\nu} \frac{\left|\BH{n}{\nu}\right|}{\BJ{n}{\nu}}
\end{equation}

Together with \eqref{eq:estim01} this shows that
$$
\left|\BH{n}{\nu\frac{R}{\eps}}\right| > \frac{\eps}{R}\frac{0.58}{\nu^{1/6}} \frac{\left|\BH{n}{\nu}\right|}{\BJ{n}{\nu}}  \sup_{x>0} |\BJ{n}{x}|.
$$

Let us now turn to \eqref{eq:bdb2}.
It is known (see \cite{QU-WONG-99}) that
\begin{equation}\label{eq:cn}
\alpha_{\nu,1}= \nu + c(\nu)\nu^{1/3}
\mbox{ where }
A_1<c(\nu) < A_1+\frac{3}{10} A_1^2\nu^{-2/3}.
\end{equation}
Where $A_1$ is a universal constant, $A_1\approx1.855757082$. 

Let us first assume that $\frac{R}{\lambda\eps}\left(1+ M\left(\frac{c(\nu)}{\nu^{1/2}}\right)\frac{c(\nu)}{\nu^{1/2}}\right)\leq1$, where $M$ is given by
\begin{equation}\label{eq:defmnu}
 M(x)= 2 + 2x + x^{5/4}.
\end{equation}
Then, using \eqref{eq:cn} we find  
$$
\frac{R}{\eps \lambda} \alpha_{\nu,1} \leq \nu\frac{1+c(\nu)\nu^{-2/3}}{1+  2 c(\nu) \nu^{-1/2}} \leq \nu.
$$
It is known (see \cite{PARIS-84}) that for all $x<\nu$,
$$
\TBJ{\nu}{x} \leq \left(\frac{x}{\nu}\exp\left(1-\frac{x}{\nu}\right)\right)^{\nu} \TBJ{\nu}{\nu}  
             \leq \frac{1}{2\nu^{1/3}} \left(\frac{x}{\nu}\exp\left(1-\frac{x}{\nu}\right)\right)^{\nu}
$$
Therefore
$$
\TBJ{\nu}{\frac{R}{\eps \lambda} \alpha_{\nu,1}} 
\leq \frac{1}{2\nu^{1/3}} \left(\frac{R}{\eps \lambda} \exp\left(1-\frac{R}{\eps\lambda}\right) \right)^{\nu-\nu^{5/6}}
r_\nu,
$$
where
\begin{eqnarray*}
\ln r_\nu &=& \nu^{5/6} \ln \left(\frac{R}{\eps \lambda}\right)+ \nu^{5/6}\left(1-\frac{R}{\eps \lambda}\right) 
 +\nu \ln\left(\frac{\alpha_{\nu,1}}{\nu}\right) + \frac{R}{\eps \lambda}(n- \alpha_{\nu,1})  \\
&\leq &  \nu^{5/6} \left( \ln \left(\frac{R}{\eps \lambda}\right) + \left(1+  \frac{c(\nu)}{\nu^{1/2}}\right)\left(1-\frac{R}{\eps \lambda}\right)\right)  \\
&\leq & \frac{\nu^{1/2}}{2}\left(1-\frac{R}{\eps \lambda}\right) \left( \left(1-\frac{R}{\eps \lambda}\right)^{-1} \ln \left(\frac{R}{\eps \lambda}\right) +1+  \frac{c(\nu)}{\sqrt{\nu}}\right).
\end{eqnarray*}
Since $x\to (1-x)^{-1}\ln(x)$ is increasing for $x<1$, we have
\begin{eqnarray*}
&&\left(1-\frac{R}{\eps \lambda}\right)^{-1} \ln \left(\frac{R}{\eps \lambda}\right) +1+  \frac{c(\nu)}{\sqrt{\nu}} \\
&\leq& 
-\left(1+ \frac{\sqrt{\nu}}{M\left(\frac{c(\nu)}{\nu^{1/2}}\right)\,c(\nu)}\right) \ln \left(1+ \frac{M\left(\frac{c(\nu)}{\nu^{1/2}}\right)\,c(\nu)}{\sqrt{\nu}}\right) + 1 +  \frac{c(\nu)}{\sqrt{\nu}}.
\end{eqnarray*}
Using the definition of $M$ \eqref{eq:defmnu}, we see that the right hand side of this last inequality is an explicit function of    $\nu^{-1/2} c(\nu)$, which is negative 
when $\nu^{-1/2} c(\nu) < 2.18$. Using \eqref{eq:cn}, we see that for  $\nu^{-1/2} c(\nu) < 2.16$ for all $\nu\geq 3/2$. Thus for all $\nu\geq 3/2$, there holds
$$
\TBJ{\nu}{\frac{R}{\eps \lambda} \alpha_{\nu,1}} 
\leq \frac{1}{2\nu^{1/3}} \left(\frac{R}{\eps \lambda} \exp\left(1-\frac{R}{\eps\lambda}\right) \right)^{\nu-\nu^{5/6}}.
$$
Next, we note that $x\to -\TBJ{\nu}{x}\TBY{\nu}{x}$ is minimal on $(0,\nu)$ at $x=0$ for all $\nu\geq1$, where it equals $(\pi\nu)^{-1}$.
Therefore
$$
-\TBY{\nu}{\frac{R}{\eps \lambda} \alpha_{\nu,1}}> \frac{2}{\pi\nu^{2/3}} \left(\frac{\eps \lambda}{R} \exp\left(\frac{R}{\eps\lambda}-1\right) \right)^{\nu-\nu^{5/6}},
$$ 
and, using that 
$$
\left|\BH{n}{\frac{R}{\eps \lambda} \alpha_{\nu,1}}\right| > -\sqrt{\frac{\pi}{2\frac{R}{\eps \lambda} \alpha_{\nu,1}}} \TBY{\nu}{\frac{R}{\eps \lambda} \alpha_{\nu,1}}
>-\sqrt{\frac{\pi}{2n}}\TBY{\nu}{\frac{R}{\eps \lambda} \alpha_{\nu,1}},
$$
we obtain
\begin{equation}\label{eq:interhn1}
 \left|\BH{n}{\frac{R}{\eps \lambda} \alpha_{\nu,1}}\right| 
> \sqrt{\frac{2}{\pi}} \frac{1}{\nu^{7/6}} \left( \frac{\eps \lambda}{R} \exp\left(\frac{R}{\eps\lambda}-1\right) \right)^{\nu-\nu^{5/6}}.
\end{equation}

Let us now suppose that $1> \frac{R}{\lambda\eps}\geq \left(1+ M\left(\frac{c(\nu)}{\nu^{1/2}}\right)\frac{c(\nu)}{\nu^{1/2}}\right)^{-1}$. 
Since $x\to x^{-1}\exp(x-1)$ is decreasing when $x\leq1$, we obtain an upper bound on
$$
\left( \frac{\eps \lambda}{R} \exp\left(\frac{R}{\eps\lambda}-1\right) \right)^{\nu-\nu^{5/6}}
$$
by replacing $R/(\eps\lambda)$ by its lower bound and $c(\nu)$ by its upper bound, given in \eqref{eq:cn}.
The resulting expression is an explicit increasing function of $n$, with limit $\exp( 2 A_1^2)\approx980$. It is then possible to verify by
inspection on a finite range for $\nu$ that for all $\nu>\frac{3}{2}$,
\begin{equation}
\frac{1}{\nu^{1/3}}\left( \frac{\eps \lambda}{R} \exp\left(\frac{R}{\eps\lambda}-1\right) \right)^{\nu-\nu^{1/3}} < 11,
\label{eq:bdhn2}
\end{equation}
this inequality being automatically satisfied when $\nu>10^{6}$ for example. The maximum occurs near $\nu=2838$.
On the other hand, $\frac{R}{\eps \lambda} \alpha_{\nu,1} \leq \alpha_{\nu,1}$ and $x\to\left|\BH{\nu}{x}\right|$ is decreasing, and see  e.g. \cite{NIST-10}
$$
\sqrt{\frac{\pi}{2}}\TBY{\nu}{\alpha_{\nu,1}} > \frac{1}{\sqrt{2}} \nu^{-1/3},
$$
therefore when $x\leq\alpha_{\nu,1}$.
\begin{equation}\label{eq:bdhn3}
\left|\BH{\nu}{x}\right| = \sqrt{\frac{\pi}{2 \alpha_{\nu,1}}}\TBY{\nu}{\alpha_{\nu,1}} 
> \frac{1}{n^{5/6}}\frac{1}{\sqrt{2+2\frac{c(\nu)}{\nu^{2/3}}}} >\frac{3}{2 n^{5/6}}.
\end{equation}
The bounds \eqref{eq:bdhn2} and \eqref{eq:bdhn3} show that when  $1> \frac{R}{\lambda\eps}\geq \left(1+ M\left(\frac{c(\nu)}{\nu^{1/2}}\right)\frac{c(\nu)}{\nu^{1/2}}\right)^{-1}$, 
we have
\begin{equation}\label{eq:interhn2}
 \left|\BH{n}{\frac{R}{\eps \lambda} \alpha_{\nu,1}}\right| 
> \frac{1}{15} \frac{1}{\nu^{7/6}} \left( \frac{\eps \lambda}{R} \exp\left(\frac{R}{\eps\lambda}-1\right) \right)^{\nu-\nu^{5/6}}.
\end{equation}
Combining \eqref{eq:interhn1} and \eqref{eq:interhn2} we have obtained that  \eqref{eq:interhn2} holds for all $\lambda\geq1$, all $n\geq1$ and all $R\leq \eps \lambda$.
To conclude, note that using \eqref{eq:bdjn} we have
$$
\frac{1}{15} \frac{1}{\nu^{7/6}} > 2^{-7/2} \frac{1}{(2\nu)^{1/3}} \sup_{x>0} \left|\BJ{n}{x} \right|.
$$
\end{proof}

\subsection{Proof of Theorem~\ref{thm:broadband}}

The proof of Theorem~\ref{thm:broadband} follows the line of the proof of the corresponding result in the two-dimensional case proved in \cite{CAPDEBOSCQ-11}.

The first step is the following proposition.
\begin{prop} \label{pro:Ito} For any $0<\tau\leq\frac{1}{4}$ and $\lambda>7$, we define 
$$
A_{n}(\tau)=\left\{ \oeps >0 \mbox{ such that } \left|\Rn \BH{n}{\oeps} \right|\leq \frac{9}{2\,\tau} \sup_{x>0} |\BJ{n}{x}|\right\},
$$
 and
$$
 B_{n}(\tau)=(0,\infty)\setminus A_{n}(\tau),
$$
then 
$$
B_{n}(\tau)\subset \bigcup_{k\in K(\lambda,n)} I_{\nu,k}(\tau),
$$
where $I_{\nu,k}(\tau)$ is defined by
\begin{equation}\label{eq:def-itau}
I_{\nu,k}(\tau):=\left\{ x\in U_{\nu,k} \mbox{ such that }\left|g_{\nu}(\lambda x)+k_{\nu}(x)\right|\leq\tau\left|k_{\nu}(x)\right|\right\},
\end{equation}
and where $K(\lambda,n)$ is the set of all positive $n$ such that $\alpha_{\nu,k}^{(1)}< \nu \lambda$. Furthermore,
\begin{equation}\label{eq:bqme}
\left| B_{n}(\tau) \right| \leq 4 \tau \frac{ 2\nu \ln \lambda}{\lambda}. 
\end{equation}
When $n=0$, the same result holds for $\tau<\frac{3}{4}$.
\end{prop}
\begin{proof}
First, note that Lemma~\ref{lem:7.1} shows that $(0,\alpha_{\nu,1}^{(1)}/\lambda)\subset A_{\nu,k}(\tau)$.  

Furthermore, we have shown in the proof of Theorem~\ref{thm:estim-all} that when $\oeps\geq\alpha_{\nu,1}^{(1)}$, 
$\left|\Rn \BH{n}{\oeps} \right|\leq 5/2 \sup_{x>0} |\BJ{n}{x}|$, thus $(\alpha_{\nu,1}^{(1)},\infty) \subset A_{\nu,k}(\tau)$. 

Thirdly, using the bound $|\Rn|\leq1$, we see that when $\nu<\oeps<\alpha_{\nu,1}^{(1)}$, we have
$$
\frac{\left|\Rn \BH{n}{\oeps} \right|^2}{\sup_{x>0} |\BJ{n}{x}|^2} \leq \frac{\left|\BH{n}{\oeps} \right|^2}{|\BJ{n}{\oeps}|^2} 
= 1 + \frac{\TBY{\nu}{\oeps}^2}{\TBJ{\nu}{\oeps}^2}.
$$
On $\nu<\oeps<\alpha_{\nu,1}^{(1)}$, ${\TBY{\nu}{\oeps}^2}/{\TBJ{\nu}{\oeps}^2}$ is decreasing, and ${\TBY{\nu}{\nu}^2}/{\TBJ{\nu}{\nu}^2}$ is a decreasing function
of $\nu$. Therefore
$$
\frac{\left|\Rn \BH{n}{\oeps} \right|^2}{\sup_{x>0} |\BJ{n}{x}|^2} \leq  1 + \frac{\TBY{\frac{1}{2}}{\frac{1}{2}}^2}{\TBJ{\frac{1}{2}}{\frac{1}{2}}^2} < 5,
$$
and we have obtained that
$$
B_{n}(\tau) \subset \left(\alpha_{\nu,1}^{(1)}/\lambda,\nu \right).
$$
Next, it is known (see \cite{CAPDEBOSCQ-11}, Proposition~8.3) that when $\oeps\in(0,\nu)$ and $\oeps \not\in I_{\nu,k}(\tau)$ for some $k$, then 
$x\in A_{n}(\tau)$ when $n\geq1$. 

The argument is simple. It turn out that by a simple calculus argument using the formula for $r_n$, 
when $g_\nu(x)>0$, $k_\nu(x)>0$, and $k_\nu(x) > \frac{2}{5}g_\nu(x)$, $|g_\nu(\lambda x) + k_\nu(x)|> \tau k_\nu(x)$, then  $x\in A_{n}(\tau)$. 
When $x \in (0,\nu)$, $g_\nu(x)>0$, $k_\nu(x)>0$, and $k_\nu(x) > \frac{2}{5}g_\nu(x)$ therefore the inclusion holds.
Since $g_{\frac{1}{2}}(x)>0$, $k_{\frac{1}{2}}(x)>0$, and $k_{\frac{1}{2}}(x) > g_{\frac{1}{2}}(x)$ on $(0,\frac{1}{2})$, the same is true when $n=0$.

The proof of Proposition~\ref{pro:Ito} will be complete once estimate \eqref{eq:bqme} is established, for $\lambda\geq7$. Since it is proved 
in \cite{CAPDEBOSCQ-11} Proposition 8.2 for $n\geq1$, we  only need to consider the case $n=0$, and
$\oeps\in(\alpha_{\frac{1}{2},1}^{(1)}/\lambda,\frac{1}{2})$. 
We have 
$$
g_{\frac{1}{2}}(x) = 1 + \frac{2x}{\tan x}\mbox{ and }k_{\frac{1}{2}}(x) = 1 + 2x\tan x>1 + 2x^2,
$$
and $\alpha_{\frac{1}{2},k} = k\pi$.
Introducing
\begin{equation}\label{eq:def-phin}
\phi_0:= \begin{array}[t]{rcl}
   (0,\frac{\pi}{2})\setminus \cup_k \{k\pi/\lambda\} &\to& \mathbb{R} \\
   x     &\to& \displaystyle \frac{g_{\frac{1}{2}} \left(\lambda x\right)}{k_{\frac{1}{2}} (x)}, 
\end{array}
\end{equation}
we have
$$
\phi_0(I_{\frac{1}{2},k}(\tau)) = [-1-\tau,-1+\tau].
$$
We first verify that $\phi_{\frac{1}{2}}$ is one-to-one on $I_{\frac{1}{2},k}(\tau)$, for $\tau$ small enough and $\lambda$ large enough.
Differentiating $\phi$ we find
$$
\phi^\prime(x) = \frac{2 x}{ k_{\frac{1}{2}}(x) } (1 -\lambda^2) 
+ \frac{g_{\frac{1}{2}}(\lambda x) + k_{\frac{1}{2}}(x)}{k^2_{\frac{1}{2}}(x)}
\left( \frac{1-g_{\frac{1}{2}}(\lambda x)  k_{\frac{1}{2}}(x)}{2x} -2x \right).
$$
When $-g_{\frac{1}{2}}(\lambda x)>k_{\frac{1}{2}}(x)$, we have 
$$
\frac{1-g_{\frac{1}{2}}(\lambda x)  k_{\frac{1}{2}}(x)}{2x}-2x > \frac{2+4x^4+4x^2}{2x} - 2x = \frac{1+2x^4}{2x}>0,
$$
therefore 
$$
\phi^\prime(x) < \frac{2 x}{ k_{\frac{1}{2}}(x) } (1 -\lambda^2).
$$ 
When $-g_{\frac{1}{2}}(\lambda x)<k_{\frac{1}{2}}(x)$ and $x\in I_{\frac{1}{2},k}(\tau)$, we have $0< g_{\frac{1}{2}}(\lambda x)+k_{\frac{1}{2}}(x)<\tau k_{\frac{1}{2}}(x)$,
and 
$$
\frac{g_{\frac{1}{2}}(\lambda x) + k_{\frac{1}{2}}(x)}{k^2_{\frac{1}{2}}(x)}
\left( \frac{1-g_{\frac{1}{2}}(\lambda x)  k_{\frac{1}{2}}(x)}{2x} -2x \right) 
< \frac{\tau }{x k_{\frac{1}{2}}(x)}\frac{ 1-4x^2 +k^2_{\frac{1}{2}}(x)}{2} 
< \frac{6\tau}{5} \frac{ x \lambda^2}{k_{\frac{1}{2}}(x)},
$$
since $1 < \alpha_{\frac{1}{2},1}^{(1)} < \lambda x$. Finally, $\lambda^2 < 1.05 (\lambda^2-1)$ when $\lambda\geq7$, thus for any $\tau\leq3/4$, we have obtained that
$$
\phi^\prime(x) < \frac{x}{ k_{\frac{1}{2}}(x) } (1 -\lambda^2)<0,
$$ 
for all $x\in I_{\frac{1}{2},k}(\tau)$.
In particular, if $I_{\frac{1}{2},k}(\tau)=[\zeta_k,\eta_k]$, we have 
$$
2\tau =\int_{\eta_k}^{\zeta_k} - \phi^\prime(t) dt \geq (\lambda^2-1)   \int_{\eta_k}^{\zeta_k} \frac{t}{ k_{\frac{1}{2}}(t) }dt = (\lambda^2-1)|I_{\frac{1}{2},k}(\tau)|
\frac{1}{\zeta_k - \eta_k} \int_{\eta_k}^{\zeta_k} \frac{t}{ k_{\frac{1}{2}}(t) }dt.
$$ 
Therefore 
$$
\sum_{k\in K(\lambda,0)} |I_{\frac{1}{2},k}(\tau)| \leq \frac{2\tau}{\lambda^2-1} \sum_{k\in K(\lambda,0)} (\zeta_k - \eta_k) \left(\int_{\eta_k}^{\zeta_k} \frac{t}{ k_{\frac{1}{2}}(t) }dt\right)^{-1}.
$$
since $x/k(x)$ is increasing on $(0,1/2)$, and $\eta_k > \alpha_{\frac{1}{2},k}^{(1)}/\lambda$,
\begin{eqnarray*}
\sum_{k\in K(\lambda,0)} (\zeta_k - \eta_k) \left(\int_{\eta_k}^{\zeta_k} \frac{t}{ k_{\frac{1}{2}}(t) }dt\right)^{-1} 
&\leq& \sum_{k\in K(\lambda,0)}  k_{\frac{1}{2}}\left(\frac{\alpha_{\frac{1}{2},k}^{(1)}}{\lambda}\right) \frac{\lambda}{\alpha_{\frac{1}{2},k}^{(1)}} \\
&\leq& \max_{k\in K}\frac{\lambda}{\alpha_{\frac{1}{2},k}^{(1)} - \alpha_{\frac{1}{2},{k-1}}^{(1)}}\int_{\alpha_{\frac{1}{2},0}^{(1)}/\lambda}^{\frac{1}{2}} \frac{1}{t}k_{\frac{1}{2}}(t) dt,
\end{eqnarray*}
where $\alpha_{\frac{1}{2},0}^{(1)}=\frac{1}{2\cos^2(\frac{1}{2})}$ 
 is a convenient choice (but any number smaller than  $\alpha_{\frac{1}{2},1}^{(1)}$ and greater than zero would do to write the Riemann sum)
.
The distance between two distinct positive solutions of $\tan(x)=2x$ is at least $\pi$, therefore
$$
\max_{k\in K}\frac{\lambda}{\alpha_{\frac{1}{2},k}^{(1)} - \alpha_{\frac{1}{2},{k-1}}^{(1)}}
 = \frac{\lambda}{\alpha_{\frac{1}{2},1}^{(1)} - \alpha_{\frac{1}{2},0}^{(1)}} < 2 \frac{\lambda^2-1}{\lambda},
$$
since $\lambda\geq7$, and
$$
\int_{\alpha_{\frac{1}{2},0}^{(1)}/\lambda}^{\frac{1}{2}} \frac{1}{t}k_{\frac{1}{2}}(t) dt = \ln(\lambda) +2 \ln(\cos(\frac{\alpha_{\frac{1}{2},0}^{(1)}}{\lambda})) < \ln(\lambda).
$$
We have obtained that
$$
\sum_{k\in K(\lambda,0)} |I_{\frac{1}{2},k}(\tau)| \leq \frac{4 \tau \ln \lambda}{\lambda},
$$
and this concludes our proof.
\end{proof}

The second step is to use Proposition~\ref{pro:Ito} to derive an estimate for the scattered field, for large enough contrast.
\begin{lem}\label{lem:highcontrast-Lemma}
Suppose $\lambda>7$. Let $\eta_{\max}$ be the following decreasing function of the contrast
\begin{equation}\label{eq:etamax}
\eta_{\max} =  \frac{5}{4}  \frac{\ln\lambda}{\lambda}.
\end{equation}
Given $\alpha>0$, for any $\eta>0$ such that
$$
\eta \leq \frac{1}{\alpha} \eta_{\max} 
$$ 
there exists a set $I$ depending on $\eta,\alpha,\eps$ and $\lambda$ such that  
$$
|I|<\frac{\eta}{\eps}
$$ 
and, for any $R\geq \eps$
$$
\sup\limits_{\sqrt{q_0}\omega \in(0,\infty)\,\setminus \, I} \left\Vert  u_{\eps}^{s}\left(|x|=R\right) \right\Vert _{H^{\sigma}} 
\leq 
18 \,\frac{\eps}{R} \,\frac{\eta_{\max}}{\eta \alpha} \,\mathcal{N}^{\sigma + 2 + \alpha}\left(u^i\right).
$$
\end{lem}
\begin{proof}
Recall that we established in \eqref{eq:pf21-1}, that for all $R\geq\eps$ we have
$$
\left\Vert  u_{\eps}^{s}\left(|x|=R\right) \right\Vert^2_{H^{\sigma}} \leq   \frac{\eps^2}{R^2}
\sum_{n=p_0}^{\infty} \sum_{m=-n}^{n} \left| a_{n,m} \,(2\nu)^{\sigma} \Rn \BH{n}{\oeps}\right|^2  
$$
For $0<\alpha\leq1$, and $\eta< \frac{1}{\alpha} \eta_{\max}$, let $\tau_n$ be given by
$$
\tau_n = \frac{1}{4(2\nu)^{2+\alpha}} \frac{\alpha \eta}{\eta_{\max}}.
$$
If $\oeps \not\in \bigcup_{n=1}^\infty B_{n,\tau_n}$, we have thanks to Proposition~~\ref{pro:Ito}
\begin{eqnarray*}
&&\sum_{n=p_0}^{\infty} \sum_{m=-n}^{n} \left|a_{n,m}\,(2\nu)^{\sigma} \Rn \BH{n}{\oeps}\right|^2  \\
&\leq& \sum_{n=0}^{\infty} \sum_{m=-n}^{n} \left|a_{n,m} 
(2\nu)^{\sigma} \frac{9}{2\,\tau_{n}} \sup_{x>0} |\BJ{n}{x}|\right|^2 \\
&=& \left(\frac{18\eta_{\max}}{\eta\alpha}\right)^2 \mathcal{N}^{\sigma+2+\alpha}(u^i)^2.
\end{eqnarray*}
From Proposition~\ref{pro:Ito}, we also know that
$$
|\bigcup_{n=1}^\infty B_{n,\tau_n}|\leq \frac{4 \ln \lambda}{\lambda} \sum_{n=0}^\infty (2\nu) \tau_n 
= \frac{\ln \lambda}{\lambda} \frac{\eta}{\eta_{\max}} \sum_{n=0}^\infty \frac{\alpha}{(2n+1)^{1+\alpha}}
\leq \eta \frac{\pi^2}{10}\leq \eta.
$$
To conclude, note that the set of excluded frequencies for $\sqrt{q_0}\omega$ is $\frac{1}{\eps} \bigcup_{n=1}^\infty B_{n,\tau_n}$.
\end{proof}
We can now conclude the proof of Theorem~\ref{thm:broadband}. 
\begin{proof}[Proof of Theorem~\ref{thm:broadband}]
When $0<\lambda \leq 1$, Theorem~\ref{thm:estim-all} implies Theorem~\ref{thm:broadband}, with $I=\emptyset$. 
When $1<\lambda\leq\eps^{-2/3}$. Theorem~\ref{thm:estim-all} shows that for all $R\geq \eps^{1/3}$, we have
$$
\sup\limits_{\omega>0} \left\Vert u_{\eps}^{s} \left(|x|=R\right) \right\Vert _{H^{\sigma}} 
\leq 
2^{4/3} \frac{\eps^{1/3}}{R} \,\mathcal{N}^{\sigma} \left(u ^{i} \right),
$$
so we can again select $I=\emptyset$. Suppose now $\lambda=\eps^{s}$, with $s>2/3$. Then $\lambda>7$,  and we can apply Lemma~\ref{lem:highcontrast-Lemma}. 
Choosing 
$$
\eta=\frac{3}{2}\,s \,\eps^{s+\frac{2}{3}}\,\lneps,
$$
we have for all $0<\alpha\leq1$,
$$
\eta < \frac{1}{\alpha} \eta_{\max},
$$
and there exists a set $I$ depending on $\lambda, \alpha$ and $\eps$ such that
for any $R\geq \eps$,
$$
\sup\limits_{\sqrt{q_0}\omega \in(0,\infty)\,\setminus \, I} \left\Vert  u_{\eps}^{s}\left(|x|=R\right) \right\Vert _{H^{\sigma}} 
\leq 
\frac{15}{\alpha} \,\frac{\eps^{1/3}}{R} \,\mathcal{N}^{\sigma + 2 + \alpha}\left(u^i\right).
$$
The size of the set $I$ is bounded by
$$
|I|\leq \frac{3}{2} s \eps^{s-1/3} \lneps \leq \eps^{1/3} \lneps,
$$
since $s\to s\eps^{s-1/3}$ is decreasing when $s\geq2/3>\lneps^{-1}$.
\end{proof}

\section*{Acknowledgements }
This work was completed in part while George Leadbetter and Andrew Parker  were visiting OxPDE during a summer undergraduate research internship
awarded by OxPDE in 2011, and they would like to thank the Centre for the wonderful time they had there.

% \bibliographystyle{amsalpha}
% \bibliography{../TeX/Mybib}

\providecommand{\bysame}{\leavevmode\hbox to3em{\hrulefill}\thinspace}
\providecommand{\MR}{\relax\ifhmode\unskip\space\fi MR }
% \MRhref is called by the amsart/book/proc definition of \MR.
\providecommand{\MRhref}[2]{%
  \href{http://www.ams.org/mathscinet-getitem?mr=#1}{#2}
}
\providecommand{\href}[2]{#2}

\end{document}